\pdfoutput=1
\documentclass{amsart}
\usepackage{booktabs}
\usepackage{amssymb}
\makeindex

\usepackage{hyperref}
\hypersetup{
    colorlinks,%
    citecolor=blue,%
    filecolor=blue,%
    linkcolor=blue,%
    urlcolor=blue
}

 \usepackage[all]{xy}
\usepackage{tikz}
\usepackage{todonotes}
\usepackage[backrefs,lite]{amsrefs}
\usepackage{enumerate}
\usepackage{verbatim}
\usepackage{multirow}
\usepackage{pgf,tikz}
\usetikzlibrary{arrows}
\usepackage{tkz-graph}%needed for graphs
\usepackage{mathtools}

\CompileMatrices

\newtheorem{theorem}{Theorem}[section]
\newtheorem{lemma}[theorem]{Lemma}

\newtheorem{proposition}[theorem]{Proposition}
\newtheorem{corollary}[theorem]{Corollary}
\theoremstyle{definition}
\newtheorem{definition}[theorem]{Definition}
\newtheorem{example}[theorem]{Example}

\newtheorem{remark}[theorem]{Remark}

\def\cc{{\mathbb C}}

\def\zz{{\mathbb Z}}
\def\rr{{\mathbb R}}

\def\qq{{\mathbb Q}}
\def\pp{{\mathbb P}}

\def\Osh{{\mathcal O}}

\def\Spec{\operatorname{Spec}} 
\def\rk{{\rm rk}}

\def\Cl{\operatorname{Cl}}

\begin{document}
\title[]{Quasismooth hypersurfaces in toric varieties}

\author{Michela Artebani}
\address{
Departamento de Matem\'atica, \newline
Universidad de Concepci\'on, \newline
Casilla 160-C,
Concepci\'on, Chile}
\email{martebani@udec.cl}

\author{Paola Comparin}
\address{
Departamento de Matem\'atica y Estad\'istica, \newline
Universidad de La Frontera, \newline
Av. Francisco Salazar 1145,
Temuco, Chile}
\email{paola.comparin@ufrontera.cl} 

\author{Robin Guilbot}
\address{Faculty of Mathematics, \newline
Computer Science and Mechanics,\newline
University of Warsaw, \newline
ul. Banacha 2, 02-097 Warszawa, Poland }
\email{rguilbot@math.cnrs.fr}


\thanks{The first author has been partially 
supported by Proyecto Fondecyt Regular N. 1130572 
and Proyecto Anillo ACT 1415 PIA Conicyt, 
the second author has been partially 
supported by Proyecto Fondecyt Postdoctorado N. 3150015 and Proyecto Anillo ACT 1415 PIA Conicyt, 
the third author was supported by the NCN project 2013/08/A/ST1/00804.}

\begin{abstract} 
We provide a combinatorial characterization of monomial linear systems on toric varieties 
whose general member is quasismooth. This is given  both in terms of the Newton polytope 
and in terms of the matrix of exponents of a monomial basis.
\end{abstract}
\maketitle

\section*{Introduction}
Let $X$ be a normal projective toric variety over $\cc$ with 
homogeneous coordinate ring $R(X)$.
According to Cox's construction, there exists a GIT quotient  
$p:\hat X\to X$ by the action of a quasi-torus, 
where $\hat X$ is an open subvariety of $\Spec R(X)$ 
obtained removing a closed subset of codimension at least two, 
the irrelevant locus \cite[Chapter 5]{CLS}.
This description allows to describe the geometry of $X$ 
by means of homogeneous coordinates, as for the usual projective space.
In particular, a hypersurface $Y$ of $X$ can be defined as (the image by $p$ of) 
the zero set of a homogeneous element $f\in R(X)$.
Such hypersurface $Y$ is  called {\em quasismooth} if $V(f)$ is smooth in $\hat X$,
i.e. the singular locus of $f$ is contained in the irrelevant locus.
This implies that the singularities of $Y$ are induced by those of the ambient space, 
that is they are due to the isotropy in the action of the quasi-torus on $\hat X$.
In particular $Y$ is smooth if $X$ is smooth.

To the authors' knowledge, the concept of quasismoothness 
first appeared in the work of Danilov \cite{Da2} and later in the work of Batyrev and Cox \cite{BC} 
for simplicial toric varieties. Under such hypothesis $X$ has abelian quotient singularities and $p$ 
is a geometric quotient. In \cite[Proposition 3.5]{BC} it is proved that if $X$ is simplicial, 
then a hypersurface $Y$ is quasismooth if and only if it is a $V$-submanifold of $X$. 
In particular $Y$ has abelian quotient singularities as well. 
In the general case, by \cite{Bo}, quasismooth hypersurfaces have rational singularities.

Our motivation for the study of this regularity condition is the fact that 
quasismooth hypersurfaces with trivial canonical class give examples of 
Calabi-Yau varieties with canonical singularities \cite{ACG}.
In fact, quasismoothness appears as a regularity condition in 
the literature on mirror symmetry, for example it is required 
for the Berglund-H\"ubsch-Krawitz duality \cite{BH}, \cite{Kr}. 

A characterization of quasismooth hypersurfaces in weighted projective spaces 
is given in \cite{kreuzerskarke} (see also \cite[Theorem 8.1]{F}).  As a consequence, this gives a nice description  
of quasismooth polynomials having the same number of monomials as variables 
(called of {\em Delsarte type}), see Corollary \ref{cor-simpl}.

In this paper we define and study quasismoothness for hypersurfaces in any normal toric variety $X$. 
Given a monomial linear system $\mathcal L$ on $X$,  we provide combinatorial conditions 
for the quasismoothness of a general member $Y$ of $\mathcal L$
in terms both of its Newton polytope and of the matrix of exponents of a defining 
equation of $Y$ in homogeneous coordinates.

The paper is organized as follows. In section \ref{sec-pre} we provide some background on the Cox construction 
for toric varieties and on monomial linear systems.
Section \ref{sec-qs} contains the definition and some preliminary observations on quasismoothness.
In section \ref{sec-newt} we characterize quasismoothness in terms of a geometric condition on 
the Newton polytope of the linear system, based on results by Khovanskii.
In section \ref{sec-matrix} we characterize quasismoothness in terms of the matrix of exponents of the linear system,
translating the previous result into a linear algebra condition.
In section \ref{sec-wps} we recover the known results on quasismoothness in fake weighted projective spaces.
In section \ref{sec-cy} we use quasismoothness to define families of Calabi-Yau hypersurfaces associated 
to good pairs of polytopes (see \cite{ACG}) and we discuss the behavior of quasismoothness under polar duality.
Finally, section \ref{sec-app} contains some applications of the previous results for hypersurfaces of low dimension.

\vspace{0.3cm}

\noindent {\em Acknowledgments.}
We would like to thank  Antonio Laface for several enlightening discussions and the anonymous referee for 
his/her comments.

\section{Preliminaries} \label{sec-pre}
In this section we recall some basic facts about the Cox ring of a toric variety 
and we recall the construction due to D.~Cox which presents any  toric variety 
as a GIT quotient of an open subset of an affine space by the action of a quasi-torus.

Let $X=X_{\Sigma}$ be a toric variety associated to a fan $\Sigma$ in $N_\qq$,
where $N$ is a lattice. 
As usual we denote by $\Sigma(1)$ the set of one dimensional cones in $\Sigma$,
which are in bijection with the integral torus-invariant divisors of $X$.
The {\em Cox ring}, or {\em homogeneous coordinate ring}, of $X$ is the  
polynomial ring 
\[
R(X)=\cc[x_\rho: \rho\in \Sigma(1)],
\] 
graded by the divisor class group $\Cl(X)$ in the natural way:
the degree of $x_{\rho}$ is the class in $\Cl(X)$ 
of the integral torus-invariant divisor corresponding to $\rho$.

Let $\bar X:={\rm Spec}\,  R(X)\cong \cc^{|\Sigma(1)|}$,
$J$ be the {\em irrelevant ideal} in $R(X)$, i.e. 
the ideal generated by the monomials 
\[
x^{\hat \sigma}:=\prod_{\rho\not\in \sigma(1)}x_\rho,\quad \sigma\in \Sigma,
\]
and $\hat X:=\bar X-V(J)$. 
The ${\rm Cl}(X)$-grading of $ R(X)$ induces an action of the quasi-torus 
 $G={\rm Spec}\, \cc[{\rm Cl}(X)]$ on $\bar X$ which preserves $\hat X$.
Moreover, there exists a morphism 
\[
p:\hat X\to X,
\] 
called {\em characteristic map}, which is a GIT quotient 
by the action of $G$ \cite[]{CLS}.
The quotient is geometric if and only if $\Sigma$ is simplicial \cite[Theorem 5.1.11]{CLS}.

Given a hypersurface $Y$ of $X$, we define its {\em pull-back} $p^*(Y)$ in $\hat X$ 
in the following way.
Let $X^0$ be the smooth locus of $X$, whose complement has codimension 
at least two in $X$ since $X$ is normal, and let $Y^0=Y\cap X^0$.
Since $Y^0$ is a Cartier divisor one can consider the usual pullback $(p_{|p^{-1}(X^0)})^*(Y^0)$ 
obtained composing  the local defining functions of $Y^0$ with $p$.
We define $p^*(Y)$ to be the closure of $p^*(Y^0)$ in $\hat X$.

Given a linear system $\mathcal L$ on $X$ let $p^*\mathcal L$ be its pull-back,
obtained taking the pull-back of each element of $\mathcal L$.
Observe that  $p^*\mathcal L$ is a linear system on $\hat X$ since 
the complement of $p^{-1}(X^0)$ has codimension at least two in $\hat X$.
Moreover $p^*\mathcal L$ is associated to a homogeneous subspace $V_{\mathcal L}$ of $R(X)$:
\[
p^*\mathcal L=\{{\rm div}(f): f\in V_{\mathcal L} \}.
\]
Let $B(\mathcal L)$ and $B^*(\mathcal L)$ be the base loci of $\mathcal L$ and 
$p^*\mathcal L$ respectively.  Clearly $p(B^*(\mathcal L))\subseteq B(\mathcal L)$.

A linear system $\mathcal L$ on $X$ will be called {\em monomial linear system} 
if the associated subspace $V_{\mathcal L}\subset R(X)$ is generated by monomials of the same degree in the variables $x_\rho, \rho\in\Sigma(1)$.

We define the  {\em Newton polytope} $\Delta(\mathcal L)$ of a monomial linear system $\mathcal L$
to be the convex hull  of the set of exponents of a monomial basis of $V_{\mathcal L}$ 
in $\rr^{|\Sigma(1)|}$.  Observe that $\Delta(\mathcal L)$ is contained in the intersection of the positive 
orthant with the affine space $Q^{-1}(w)$, where $w\in \Cl(X)$ is the degree of 
the elements in $V_{\mathcal L}$ and 
\[
Q:\rr^{|\Sigma(1)|}\to \Cl(X)
\]
 is the homomorphism which associates 
to $e_\rho$ the degree of $x_\rho$.

If $\mathcal L$ is a monomial linear system, then 
the base locus of $p^*\mathcal L$ is the vanishing locus of all the monomials generating $V_{\mathcal L}$.
In particular it is a union of strata of the following form
\[
D_{\sigma}=\{x\in \hat X: x_{\rho}=0, \rho\in \sigma(1)\},
\]
where $\sigma\in \Sigma$. The following result implies that monomial linear systems 
with the same Newton polytope have the same base locus.

\begin{lemma}\label{bs}
Let $\mathcal L$ be a monomial linear system on a complete toric variety and $\Delta(\mathcal L) \subset \rr^{|\Sigma(1)|}$ be 
its Newton polytope.
Then $B^*(\mathcal L)$ is the vanishing locus of all the monomials $x^a=\prod_{\rho \in \Sigma(1)} x_{\rho}^{a_\rho}$, where 
$a$ is a vertex of $\Delta(\mathcal L)$.
\end{lemma}

\begin{proof}
Let $S$ be the set of exponents of a monomial basis of $V_{\mathcal L}$. 
It is sufficient to show that the vanishing locus of all the monomials $x^b$ for $b \in S$ 
is the vanishing locus of all the monomials $x^a$ for $a$ a vertex of $\Delta(\mathcal L)$. 
One inclusion is obvious, let us prove the other. For all $b \in S$ there exist vertices $a_1, \ldots, a_k$ of $\Delta(\mathcal L)$ 
and positive rational numbers $\lambda_1, \ldots, \lambda_k \in \ (0,1] \cap \qq$ such that $\sum_{i=1}^k \lambda_i a_i=b$ 
and $\sum_{i=1}^k \lambda_i =1$. It follows that all the variables appearing in the $x^{a_i}$'s also appear in $x^b$. 
So the vanishing of all the $x^{a_i}$'s implies the vanishing of $x^b$, which proves the second inclusion.
\end{proof}

\section{Quasismoothness}\label{sec-qs}
\begin{definition}
Let $X=X_\Sigma$ be a toric variety and $p:\hat X\to X$ be its characteristic map.

A hypersurface $Y\subset X$ is  {\em quasismooth at a point} $y\in Y$ if $p^*(Y)$ is smooth 
at $p^{-1}(y)\cap p^*(Y)$; 
it is {\em quasismooth}
if $p^*(Y)\subset \hat X$ is smooth. 
\end{definition}

Observe that $p^*(Y)$ is the zero set of a homogeneous element $f_Y\in R(X)$ in $\hat X$,
since it is $G$-invariant. We will denote by $S(f_Y)$ the singular locus of $f_Y$ in $\bar X$,
that is the zero set of al partial derivatives of $f_Y$ in $\bar X$.
Thus $Y$ is quasismooth if and only if $S(f_Y)$ is contained
in the irrelevant locus $V(J)$ of $X$.

We start considering the behaviour of quasismoothness in linear systems.

\begin{proposition} \label{zariski}
Let $\mathcal L$ be a linear system of a complete toric variety $X_\Sigma$.
The set of quasismooth elements in $\mathcal L$ is an open Zariski subset 
of $\mathcal L$.   
\end{proposition}

\begin{proof}
Let $p:\hat X\to X$ the characteristic map of $X$ and let $r=|\Sigma(1)|$. 
We will denote by $f_Y\in R(X)$ a defining polynomial for a hypersurface $Y$ of $X$,
as explained above.
Consider the set 
\[
S=\{(x,Y)\in \hat X\times \mathcal L: f_Y(x)=\frac{\partial f_Y}{\partial x_j}(x)=0,\ j=1,\dots,r\}
\]
and the projection $\pi:S\to \mathcal L$. The fiber over $Y\in \mathcal L$ is empty 
exactly when $Y$ is quasismooth. Observe that $\pi$ factors through 
$\bar\pi: (p\times{\rm id})(S)\to \mathcal L$.
Since $ (p\times{\rm id})(S)$ is complete, then the image of $\bar\pi$ is Zariski closed in $\mathcal L$  \cite[Exercise 4.4, Chapter II]{Ha}.
This gives the statement.
\end{proof}
As a direct consequence of the previous Proposition, one has the following.

\begin{corollary} Let  $X=X_{\Sigma}$ be a complete toric variety and $\mathcal L$ be a  
linear system on $X$.
Then the following are equivalent:
\begin{enumerate}
\item the general  element of $\mathcal L$ is quasismooth;
\item there exists a quasismooth element in $\mathcal L$.
\end{enumerate}
\end{corollary}

In what follows we will say that a linear system is quasismooth if its general element is quasismooth.

\begin{proposition}
The general element of $\mathcal L$ is quasismooth outside of $p(B^*(\mathcal L))$.
In particular $\mathcal L$ is quasismooth if $B(\mathcal L)=\emptyset$.
\end{proposition}

\begin{proof}
By the first Bertini's theorem applied to the linear system $p^*\mathcal L$ on $\hat X$,
we have that the singular locus of the general element of $p^*\mathcal L$ is contained 
in its base locus. Taking the image in $X$, this gives the statement.
\end{proof}

Finally, we observe that quasismoothness behaves well with respect to finite coverings which are the identity 
in Cox coordinates (see \cite[Lemma 1.1]{ACG}).

\begin{proposition}\label{fq}
Let $X$ be a toric variety associated to a fan $\Sigma\subset N_\qq$,
$\iota: N\to N'$ be a lattice monomorphism with finite cokernel and $\pi : X \to X'$ 
be the associated finite quotient. If the primitive generators  in $N$ 
of the rays of the fan of $X$ are primitive in $N'$, 
then the homomorphism $\pi^*: R(X') \to R(X)$ can be taken to be the identity.
In particular, if $\mathcal L$ is a monomial linear system on $X'$, then 
$\mathcal L$ is quasismooth if and only if $\pi^*\mathcal L$ is quasismooth.
\end{proposition}

\section{In terms of the Newton polytope}\label{sec-newt}
In this section we provide a first characterization of quasismoothness 
based on results in \cite{K}.
We will denote by $T^r\cong (\cc^*)^r$  the $r$-dimensional torus 
with coordinates $x=(x_1,\dots,x_r)$.
A Laurent polynomial on $T^r$ can be written as a finite sum 
$P=\sum_{m}a_mx^m$ with $m\in \zz^r$ and $a_m\in \cc$.
The {\em support} of $P$ is the set of $m\in \zz^r$ such that $a_m\not=0$ and
the convex hull of the support of $P$ in $\rr^r$ is the {\em Newton polytope} $\Delta(P)$ of $P$.

\begin{definition}
A collection of $\ell>0$ non-empty polytopes in $\rr^n$ is {\em degenerate} 
if it is possible to  
translate all the polytopes in an $(\ell-1)$-dimensional subspace.
A collection of polytopes in $\rr^n$ is called {\em dependent} if it contains a degenerate subcollection.
By convention, the empty collection of polytopes is independent.
\end{definition}

Given a set of Laurent polynomials $P_1,\dots,P_k\in \cc[x_1,x_1^{-1},\dots,x_r,x_r^{-1}]$,
we will say that they are {\em general with fixed support} if each $P_i$ is general in the family 
of polynomials having the same support as $P_i$.
The following comes from \cite[Theorem 1, \S 2.5]{K} and \cite[Lemmas 1,2, \S 2.2]{K}.

\begin{theorem}\label{khov}
Let $P_1 = \dots = P_k = 0$ be a system of equations
where $P_i\in \cc[x_1,x_1^{-1},\dots,x_r,x_r^{-1}]$ are general with fixed support.
If the Newton polyhedra $\Delta(P_1),\dots, \Delta(P_k)$ are dependent,
then the system is not compatible in $T^r$.
Otherwise it defines an analytic $(r - k)$-dimensional manifold in $T^r$.
\end{theorem}

We now consider hypersurfaces in toric varieties.
Let $\mathcal L$ be a monomial linear system on a toric variety $X=X_{\Sigma}$ with $r=|\Sigma(1)|$
and let $f\in R(X)$ be a defining element for a general hypersurface in the linear system.
The stratum $D_\sigma$, $\sigma\in \Sigma$, is contained in $B^*(\mathcal L)$ 
if and only if one can write $f=\sum_{\rho\in \sigma(1)} x_\rho f_\rho$. 
We now define a set of polytopes associated to the pair $(\mathcal L,\sigma)$.

\begin{definition} With the previous notation, we define $\Delta_\sigma^\rho(\mathcal L)$ 
to be the Newton polytope of the restriction of $f_{\rho}$ to $D_{\sigma}$.
\end{definition}

It is easy to see that the definition of $\Delta_\sigma^\rho(\mathcal L)$ does not depend 
on the choice of the writing for $f$, which is not unique in general (this also follows from the description  in 
Remark \ref{polytopes}).

\begin{remark}\label{ambient}
All polytopes $\Delta_\sigma^{\rho}(\mathcal L)$, $\rho\in \sigma(1)$, 
 are contained in the linear subspace 
 $Q^{-1}(w)\cap \{m_\rho=0: \rho\in \sigma(1)\}\subset \rr^r$ (see Section \ref{sec-pre}),
 whose dimension is equal to $\dim p(D_\sigma)$.  
Of course some of the polytopes can be empty.
\end{remark}

\begin{remark}\label{polytopes} After restricting to $D_{\sigma}$, the only left monomials in $f_{\rho}$ 
are those in the variables $x_{\tau}$ with $\tau\not\in \sigma(1)$.
Thus we have
\[
\Delta_\sigma^\rho(\mathcal L)+e_{\rho}={\rm conv}(m\in \Delta(f)\cap \zz^r: m_{\rho}=1 \text{ and } m_\gamma=0 \text{ for } \gamma\in \sigma(1)\backslash\{\rho\}).
\]
These polytopes, if not empty, are disjoint faces of the face of $\Delta(f)$ defined by $\sum_{\rho\in \sigma(1)}m_{\rho}=1$.
\end{remark}

\begin{theorem}\label{thm1}
Let $X=X_{\Sigma}$ be a toric variety  and $\mathcal L$ be a monomial linear system on $X$.  
The general element of $\mathcal L$ is quasismooth if and only if for any $\sigma\in \Sigma$ 
such that $D_\sigma\subseteq B^*(\mathcal L)$, 
the set of non-empty polyhedra $\Delta_\sigma^\rho(\mathcal L)$, $\rho\in \sigma(1)$, is dependent.
\end{theorem}

\begin{proof} 
Let $f\in R(X)$ be a defining element for a general hypersurface in the linear system.
Assume $f$ to be not quasismooth, that is the singular locus $S(f)$ of $f$ intersects 
$\hat X$. Let $S$ be an irreducible component of $S(f)\cap \hat X$ and 
let $D_\sigma$ be the smallest toric stratum containing $S$, $\sigma\in\Sigma$.
Observe that $D_\sigma \subseteq B^*(\mathcal L)$, thus we can write $f=\sum_{\rho\in \sigma(1)}x_\rho f_\rho$.
By the minimality condition on $D_\sigma$,
$S$ contains a point in the torus $T_\sigma$ of coordinates $x_\tau, \tau \not\in \sigma(1)$.
Thus 
\[S(f)\cap D_\sigma=\cap_{\rho\in \sigma(1)} V(f_\rho)\cap D_\sigma
\] 
intersects the torus $T_\sigma$.
By Theorem \ref{khov} the set of Newton polyhedra of the restrictions of the 
$f_\rho$'s to $D_\sigma$ is independent. Observe that the polynomials $f_\rho$ 
are generic with fixed support by the generality assumption on $f$.

Conversely, if there is a stratum  $D_{\sigma}$, $\sigma\in\Sigma$, such that   
$D_\sigma\subseteq B^*(\mathcal L)$ and the set of polytopes
$\Delta_\sigma^\rho(\mathcal L)$ with $\rho\in \sigma(1)$ is  independent, then by Theorem \ref{khov} 
 the $f_\rho$'s have a common zero in $T_\sigma$.
 Thus $S(f)\cap T_\sigma$ is not empty. This implies that $S(f)\cap \hat X$ is non empty as well.
\end{proof}

\begin{corollary}\label{empty}
If there exists $\sigma\in\Sigma$ such that $D_\sigma\subseteq B^*(\mathcal L)$ 
and
$\Delta_\sigma^\rho(\mathcal L)=\emptyset$ for all $\rho\in\sigma(1)$, 
then the general element of $\mathcal L$ is not quasismooth.
\end{corollary}

Given $D_\sigma\subseteq B^*(\mathcal L)$ we denote by
 $k_{\sigma}(\mathcal L)$ the number of $\rho\in \sigma(1)$ such that  ${f_\rho}_{|D_\sigma}$ is not zero,
 that is the number of non-empty polytopes $\Delta_\sigma^{\rho}(f)$, $\rho\in \sigma(1)$.

\begin{corollary}\label{cor-qs}
If $k_\sigma(\mathcal L)>\dim p(D_\sigma)$ for all $\sigma\in \Sigma$ such that $D_\sigma\subseteq B^*(\mathcal L)$, 
then the general element of $\mathcal L$ is quasismooth. 
 \end{corollary}
\begin{proof}
This is an immediate consequence of Remark \ref{ambient} and the fact that 
any set containing at least $\ell+1$ polyhedra in a $\ell$-dimensional space is dependent.
\end{proof}

\begin{example}\label{blowup}
The following example shows that the converse of Corollary \ref{cor-qs} is false.
Let us consider the linear system $\mathcal L$ in $X=\mathbb P^3$ generated by the following monomials:
\[
x_1^3x_2,\ x_2^4,\ x_3^3x_4,\ x_4^4.
\]
Consider the blow up $\tilde X$ of $X$ at the points 
$(0,0,1,0)$ and $(1,0,0,0)$.
The Cox ring of  $\tilde X$ is $\cc[y_1,\dots,y_6]$, where the degrees of the variables 
are the columns of the matrix
\[
\left(
\begin{array}{cccccc}
1 & 1 & 1 & 1 & 0 & 0\\
0 & 0 & 1 & 0 & 1 & 0\\
1 & 0 &0 & 0 & 0 & 1 
\end{array}
\right)
\]
and the variables $y_5, y_6$ define the two exceptional divisors.
Moreover, a computation using the following Magma \cite{Magma} program 

\begin{verbatim}
P<[x]> := ProjectiveSpace(Rationals(),3);
B1,f1 := Blowup(P,&+[Rays(Fan(P))[i]: i in [1,2,4]]);
B<[y]>,f2 := Blowup(B1,&+[Rays(Fan(B1))[i]: i in [2,3,4]]);
f2*f1;
\end{verbatim}
gives that the blow up map in Cox coordinates is given by
\[
(y_1,y_2,y_3,y_4,y_5,y_6)\mapsto (y_1y_5, y_2y_5y_6,y_3y_6,y_4y_5y_6).
\]
Thus an easy computation shows the proper transform $\tilde{\mathcal L}$ of $\mathcal L$ in $\tilde X$ is generated by
\[
y_1^3y_2y_5^3,\ y_2^4y_5^3y_6^3,\ y_3^3y_4y_6^3,\ y_4^4y_5^3y_6^3.
\] 
The closure of the base locus of $p^*\tilde{\mathcal L}$ is given by
\[
\{y_2=y_4=0\}\cup\{y_2=y_6=0\}\cup\{y_4=y_5=0\}.
\]
An easy check of the criterion in Theorem \ref{thm1} shows that 
the general element of $\tilde{\mathcal L}$ is quasismooth (all polytopes $\Delta_{\sigma}^{\rho}(\mathcal{\tilde L})$ are either points or empty).
On the other hand, when $D_\sigma=\{y_2=y_6=0\}$, one has $k_\sigma(\tilde{\mathcal L})=1\leq3=\dim p(D_\sigma)$.
\end{example}

\begin{example}\label{not-max-comp} 
The following example shows that  the condition in Theorem \ref{thm1} 
has to be checked for all subsets $D_\sigma$ in the base locus, not only for 
the maximal ones.
Let $X=\pp^4$ and $\mathcal L$ be the linear system 
generated by $x_1^3,\ x_2^2x_1, x_3^2x_1,\ x_2x_3x_4$.
The base locus of $\mathcal L$ is 
\[
\{x_1=x_2=0\}\cup\{x_1=x_3=0\}\cup\{x_1=x_4=0\}.
\]
Observe that $\mathcal L$ is not quasismooth since 
the polytopes $\Delta_{\sigma}^{\rho}$ are all empty 
for $D_\sigma=\{x_1=x_2=x_3=0\}$.
On the other hand, the condition of Theorem \ref{thm1} 
holds for the three maximal components of the base locus.
\end{example}

The theorem easily implies that quasismoothness only depends on the Newton polytope of $\mathcal L$.
In particular it is enough to consider the linear system generated by the monomials associated to the vertices
of the Newton polytope.

\begin{corollary}\label{vert}
Let $X$ be a toric variety and let $\mathcal L, \mathcal L'$ 
be two monomial linear systems on $X$ such that $\Delta(\mathcal L)=\Delta(\mathcal L')$.
The general element of $\mathcal L$ is quasismooth if and only if the same holds for the general element of 
$\mathcal L'$.
\end{corollary}

\begin{proof}
By Lemma \ref{bs}, $B^*(\mathcal L)=B^*(\mathcal L')$.
Moreover, given $\sigma\in \Sigma$, 
the polytopes $\Delta_\sigma^\rho(\mathcal L)$ and $\Delta_\sigma^\rho(\mathcal L')$ are the same by  Remark \ref{polytopes}. Thus we conclude by Theorem \ref{thm1}.
\end{proof}

\begin{example}\label{ex1}
Let $X=\pp^2\times\pp^1$ and $\mathcal L$ be the monomial linear system generated by 
\[ x_0^2x_1y_0^2,\ 
x_0^2x_1y_1^2,\ 
 x_0^2x_2y_0^2,\ 
 x_0^2x_2y_1^2,\ 
 x_1^3y_0^2,\ 
 x_1^3y_1^2,\ 
 x_2^3y_0^2,\ 
x_2^3y_1^2.\]
The only strata $D_\sigma$ such that $D_\sigma$ is  contained in $B^*(\mathcal L)$ 
are $D_{\sigma}=\{x_1=x_2=0\}$
and its two substrata 
\[
D_{\sigma_0}=\{x_1=x_2=y_0=0\},\quad  D_{\sigma_1}=\{x_1=x_2=y_1=0\}.
\]
Considering $\sigma$, one has that 
\[
\Delta_\sigma^{x_1}(\mathcal L)={\Delta_\sigma^{x_2}}(\mathcal L)={\rm conv}((2,0,0,2,0), (2,0,0,0,2)).
\]
Thus the two polytopes form a dependent set.
The same can be repeated for $D_{\sigma_0}$ and $D_{\sigma_1}$
so that the general element of $\mathcal L$ is quasismooth.
\end{example}

\begin{example}\label{p1}
Let $X=\pp^1\times\pp^1\times\pp^1$ with coordinates $(x_0,x_1,y_0,y_1,z_0,z_1)$. 
Let $\mathcal L$ be the linear system generated by the monomials
\[
x_1^2y_0^2z_0^2,\
x_1^2y_1^2z_0^2,\
x_1^2y_0^2z_1^2,\
x_0^2y_1^2z_1^2,\
x_1^2y_1^2z_1^2,\
x_0^2y_1^2z_0^2,\
x_0x_1y_0^2z_1^2,\
x_0x_1y_0^2z_0^2
.\] 
The base locus is given by $D_{\sigma}=\{x_1=y_1=0\}$ and 
the only non-empty polytope 
$\Delta_\sigma^\rho(\mathcal L),\rho\in\sigma(1)$ is a segment.
Thus Theorem \ref{thm1} proves that $\mathcal L$ is not quasismooth. 
\end{example}

\section{In terms of the exponents matrix}\label{sec-matrix}
Let $X = X_{\Sigma}$ be a projective toric variety, 
$\mathcal L$ be a monomial linear system on $X$  
of degree $w\in \Cl(X)$ and $\Delta(\mathcal L)$ be its Newton polytope.
We define the {\em matrix of exponents of } $\mathcal L$ to be the matrix
 $A$ whose rows are the vectors of exponents
of an ordered monomial basis of $V_{\mathcal L}$.   
 Moreover, we will denote by $A_{I,J}$ the submatrix of $A$ 
whose rows are indexed by elements of $I \subseteq \Delta(\mathcal L)\cap\zz^r$ 
and columns are indexed by elements of $J \subseteq \Sigma(1)$. 
For any such matrix we denote by $\rr A_{I,J}$ the linear span of the columns of $A_{I,J}$.
Given a non-empty subset $\gamma\subseteq \Sigma(1)$ we define 
\[
M_{\gamma}=\{m\in \Delta(\mathcal L)\cap\zz^r: \sum_{i\in \gamma} m_i=1\}. 
\]

\begin{theorem}\label{propqs}
Let $X = X_{\Sigma}$ be a projective toric variety with characteristic map $p:\hat X\to X$ and
$\mathcal L$ be a monomial linear system on $X$ with matrix of exponents $A$.
The following are equivalent
\begin{enumerate}
\item $\mathcal L$ is quasismooth;
\item for all $\sigma\in \Sigma$ such that $D_{\sigma}\subseteq B^*(\mathcal L)$ there exists a non-empty 
subset $\gamma \subseteq \sigma(1)$ 
such that $M_{\gamma}$ is not empty  and  
\begin{equation}\label{eq}
2\,\rk (A_{M_\gamma,\gamma})>\rk(A_{M_{\gamma},\Sigma(1)}).
\end{equation}
\end{enumerate}
\end{theorem}

\begin{proof}
Let $\sigma\in \Sigma$ and $\gamma\subseteq \sigma(1)$ 
as in the statement, $k:=|\gamma|$ and $s:=|M_\gamma|$.
We can assume that $\Delta_{\sigma}^{\rho}(\mathcal L)$ is not empty for all $\rho\in \gamma$ 
since the empty polytopes give zero columns in both $A_{M_\gamma,\gamma}$ 
and $A_{M_{\gamma},\Sigma(1)}$.
We will now prove that the collection of polytopes $\{\Delta_{\sigma}^{\rho}(\mathcal L)\}_{\rho\in \gamma}$ 
is degenerate if and only if  $2\,\rk (A_{M_\gamma,\gamma})>\rk(A_{M_{\gamma},\Sigma(1)})$.
This implies the thesis by Theorem \ref{thm1}.

The minimal dimension $d$ of a linear space containing translates of  the polytopes $\Delta_{\sigma}^{\rho}(\mathcal L)$,  $\rho\in \gamma$,   
is the dimension of the linear span of any set of translates of such polytopes that all contain the origin. In particular if we pick a
 lattice point $m(\rho) \in \Delta_{\sigma}^{\rho}(\mathcal L)\cap \zz^r$ for each $\rho\in \gamma$, we have 
\[
d=\dim\,{\rm Span}(\Delta_{\sigma}^{\rho}(\mathcal L)-m(\rho): \rho\in \gamma).
\]
The columns of $A_{M_\gamma ,\gamma}$ are of the form $A_\rho = (A_{m,\rho })_{m\in M_\gamma}$ with $A_{m,\rho} = 1$ if $m\in \Delta_{\sigma}^{\rho}(\mathcal L)$ and 0 otherwise.
In particular the columns of $A_{M_\gamma ,\gamma}$ are linearly independent 
since for $\rho\not= \rho'$ the polytopes $\Delta_{\sigma}^{\rho}(\mathcal L)$ and $\Delta_{\sigma}^{\rho'}(\mathcal L)$  
have no common lattice point. It follows
that $\rk(A_{M_\gamma,\gamma}) = k$ and that the linear map 
\[
\phi_{\gamma} : \rr^k \to \rr A_{M_\gamma,\Sigma(1)}\subseteq \rr^s,\quad x\mapsto A_{M_\gamma,\gamma}\cdot x
\] 
is injective.
We now show that $d=\dim({\rm coker}\,\phi_\gamma)$.
This gives the thesis since the degeneracy condition $k > d$ is thus equivalent to 
$k>\rk(A_{M_\gamma,\Sigma(1)})-k$.
We consider the following commutative diagram with exact rows
\[
\xymatrix{
0\ar[r] & \rr^k\ar@{=}[d]\ar[r]^{\tilde \phi_\gamma} & \rr^ {s}\ar[r]^{\tilde \alpha} & \tilde C\ar[r] & 0\\
0\ar[r] & \rr^k\ar[r]^-{\phi_\gamma} & \rr A_{M_\gamma,\Sigma(1)}\ar[r]^{\alpha}\ar[u]^{i} &  C\ar[u]\ar[r] & 0,\\
} 
\]
where the vertical arrows are inclusions.
In order to compute $\dim(C)$ we compute the rank of the map $\tilde \alpha \circ i$. 
 The map $\tilde\alpha$ is given by the  $(s-k)\times s$ 
matrix $B$  which is the vertical join of the matrices $B_\rho$, with $\rho\in \gamma$, where the rows of $B_\rho$ 
are indexed by the elements $m\in \Delta_{\sigma}^{\rho}(\mathcal L)\cap\zz^n-\{m(\rho)\}$ 
and have $1$ in the position $m(\rho)$, $-1$ in the position $m$ and $0$ in all other positions.

The rows of the matrix $B\cdot A_{M_\gamma,\Sigma(1)}$ are the vectors of exponents of the monomials $x^{m(\rho)}/x^{m}$ for $\rho\in\gamma$, $m\in\Delta_{\sigma}^{\rho}(\mathcal L)\cap\zz^n-\{m(\rho)\}$.
It follows that the rank of $\tilde \alpha\circ i$ is equal to the dimension of 
the linear span of the differences $m(\rho)-m$, and hence is $d$, which proves the theorem.
\end{proof}

\begin{remark} In Theorem \ref{propqs} one can replace $M_{\gamma}$ 
with the set $V$ of vertices of all the polytopes $\Delta_{\sigma}^{\rho}(\mathcal L)$, $\rho\in \gamma$.
In fact $\rk (A_{V,\gamma})=\rk (A_{M_\gamma,\gamma})=k$ and 
$\rk(A_{V,\Sigma(1)})=\rk(A_{M_{\gamma},\Sigma(1)})$ since any lattice point 
in $M_\gamma$ is in the linear span of  $V$.

\end{remark}

\begin{example}
Let $\mathcal L$ be as in Example \ref{ex1}. The matrix $A$ is of the form
\[A=\left(\begin{array}{ccccc}
2&1&0&2&0\\
2&1&0&0&2\\
2&0&1&2&0\\
2&0&1&0&2\\
0&3&0&2&0\\
0&3&0&0&2\\
0&0&3&2&0\\
0&0&3&0&2
\end{array}
\right).
\]
When considering $D_\sigma=\{x_1=x_2=0\}\subset B^*(\mathcal L)$ one can 
choose $\gamma=\{\rho_1,\rho_2\}$, where $\rho_i$ is the ray 
corresponding to the variable $x_i$, since we already observed that $\Delta_\sigma^{\rho_j}(\mathcal L), j=1,2$ are not empty. 
With this choice, $$M_\gamma= \{(2,1,0,2,0),(2,1,0,0,2),(2,0,1,2,0),(2,0,1,0,2)\}$$and 
the matrices appearing in the proof of Theorem \ref{propqs} are 
\[A_{V,\gamma}=\left(\begin{array}{cc}
1&0\\
1&0\\
0&1\\
0&1\\
\end{array}
\right),\
A_{V,\Sigma(1)}=\left(\begin{array}{ccccc}
2&1&0&2&0\\
2&1&0&0&2\\
2&0&1&2&0\\
2&0&1&0&2
\end{array}
\right),\
B=
\left(\begin{array}{cccc}
1&-1&0&0\\
0&0&1&-1
\end{array}
\right).
\]
We have $\rk(A_{M_\gamma,\gamma})= \rk(A_{V,\gamma})=2$ and $\rk(
A_{M_\gamma,\Sigma(1)})=\rk(
A_{V,\Sigma(1)})=3$, so that the condition is satisfied for this cone $\sigma$.
The same can be repeated for $D_{\sigma_0}=\{x_1=x_2=y_0=0\}$ and $D_{\sigma_1}=\{x_1=x_2=y_1=0\}$, proving thus quasismoothness of $\mathcal L$.
\end{example}

\section{The case of fake weighted projective spaces}\label{sec-wps}
In this section we  
show how our results allow to recover the 
known classification of quasismooth hypersurfaces in fake weighted projective 
spaces (see \cite{Kou1, Kou2, kreuzerskarke, F} and \cite[Remark 2.3]{HK} for further references).
A fake weighted projective space is a complete simplicial toric variety $X$ with Picard number one (see also \cite{Bu} and  \cite[Lemma 2.11]{BC}).
 In particular $X$ is $\mathbb Q$-factorial and $\bar X\backslash \hat X=\{0\}$.

We start giving an easy necessary condition for quasismoothness.
 We denote by $r_\sigma$ the dimension of  $(\bar X\backslash \hat X)\cap D_\sigma$.

\begin{proposition}\label{thm-qs}
Let $X=X_{\Sigma}$ be a projective toric variety 
and $\mathcal L$ be a quasismooth monomial linear system on $X$ 
whose monomial basis does not contain any generator of $R(X)$.
Then  for any  $D_\sigma\subseteq B^*(\mathcal L)$ one has
\begin{equation}
\dim(D_\sigma)-k_\sigma(\mathcal L)\leq r_\sigma.
\label{eq-qs}
\end{equation}
 In particular $2\dim(D_\sigma)\leq r_\sigma+|\Sigma(1)|$.  
 \end{proposition}

\begin{proof}
Let $R(X)=\cc[x_1,\ldots,x_r]$ and $D_\sigma$ be a subset of $B^*(\mathcal L)$ 
such that $\dim(D_\sigma)-k_\sigma(\mathcal L)>r_\sigma$.
Thus the general $f$ in $\mathcal L$ can be written as $f=\sum_{\rho\in \sigma(1)}x_\rho f_\rho$ and
$S(f)\cap D_\sigma$ is equal to $D_\sigma\cap V(f_\rho: \rho \in \sigma(1))$.
Observe that all $f_{\rho}$ are not constant by the hypothesis on $\mathcal L$.
Let $S$ be the closure of an irreducible component of 
$S(f)\cap D_\sigma$ and
$V=\bar X\backslash \hat X\cap S$.
We have that
\[
\dim(S)\geq \dim(D_\sigma)-k_\sigma(\mathcal L)>r_\sigma\geq \dim(V).
\]
Thus $V$ is properly contained in $S$,
so that $\mathcal L$ is not quasismooth.
The last statement follows from \eqref{eq-qs} since $k_\sigma(\mathcal L)\leq |\Sigma(1)|-\dim(D_\sigma)$.
\end{proof}

\begin{remark} Let $\mathcal L$ be a monomial linear system 
on a toric variety $X$ whose monomial basis contains a generator $x_0$ of $R(X)$.
 In this case, a general element of $\mathcal L$ 
is defined by an equation of the form $f=\alpha x_0+g$, where $\alpha$ is a constant and $g$ is a polynomial not containing $x_0$. 
Either by a direct computation or applying Theorem \ref{thm1} 
one easily shows that $\mathcal L$ is quasismooth. 
\end{remark}

\begin{corollary}\label{wps}
Let $X=X_{\Sigma}$ be an $n$-dimensional fake weighted projective space and 
$\mathcal L$ be a monomial linear system on $X$ whose monomial basis does not contain any generator of $R(X)$.
 Then  $\mathcal L$ is quasismooth if and only if  
  $\dim(D_\sigma)-k_\sigma(\mathcal L)\leq 0$ for any $\sigma\in \Sigma$ such that $D_\sigma\subseteq B^*(\mathcal L)$.
 \end{corollary}

\begin{proof}
This follows from  Corollary \ref{cor-qs} and Proposition \ref{thm-qs} 
since $r_\sigma=0$ and, since $X$ is $\qq$-factorial, 
then $\dim p(D_\sigma)=\dim(D_\sigma)-\rk\, {\rm Cl}(X)=\dim(D_\sigma)-1$.
\end{proof}

A different formulation is the following (see also \cite[Theorem 8.1]{F}).

\begin{corollary}\label{cor:wps}
Let $X=X_{\Sigma}$ be an $n$-dimensional fake weighted projective space and 
$\mathcal L$ be a monomial linear system on $X$ with matrix of exponents $A$ 
whose monomial basis does not contain any generator of $R(X)$.
Then $\mathcal L$ is quasismooth if and only if 
for any subset $\gamma\subseteq \{0,\dots,n\}$ either $A$ has a row whose entries indexed by $\gamma$ are all zero, 
or there exists a non-empty submatrix $S=A_{I,\gamma}$ of $A$ such that $\sum_{j\in \gamma}S_{ij}= 1$ for all $i \in I$ and such that $\rk (S)\geq n+1-|\gamma|$.
 \end{corollary}

\begin{proof}
Given $\gamma$ as in the statement, this identifies $D_{\gamma}=\{x_i=0: i\in \gamma\}\cap \hat X$.
Observe that $D_{\gamma}$ is not contained in the base locus of $p^*\mathcal L$ 
if and only if there exists a monomial in $V_{\mathcal L}$ not using the variables indexed by $\gamma$,
i.e. $A$ has a row whose entries indexed by $\gamma$ are all zero.
Moreover,  since $\sum_{j\in \gamma}S_{ij}= 1$ for all $i \in I$, there is exactly one non zero 
entry in each row of $S$, so that $\rk(S)$ equals the number of non zero columns, which is $k_{\gamma}(\mathcal L)$.
The proof thus follows from Corollary \ref{wps} observing that $\dim(D_{\gamma})=n-|\gamma|$.
\end{proof}

We now consider the case when the general element  of $\mathcal L$ 
is of Delsarte type, i.e. with the number of monomials equal to the number of variables.
In particular $\Delta(\mathcal L)$ is a simplex if it is full-dimensional. 
We give an alternative proof of the following known result (see  \cite{kreuzerskarke, HK}).

\begin{corollary}\label{cor-simpl}
Let $\mathcal L$ be a monomial linear system of  Delsarte type 
on a fake weighted projective space $X$ whose degree is bigger 
than the degree of each variable in $R(X)$.
Thus $\mathcal L$ is quasismooth if and only if the general element $f$of $\mathcal L$
can be written as sum of disjoint invertible polynomials of the following form 
(called atomic types):
\begin{align*}
f_{Fermat}&:=x^a,\\
f_{chain}&:=x_1^{a_1}x_2+x_2^{a_2}x_3+\ldots+x_k^{a_k},\\
f_{loop}&:=x_1^{a_1}x_2+x_2^{a_2}x_3+\ldots+x_k^{a_k}x_1,
\end{align*}
with $a,a_i>1$.
\end{corollary}

\begin{proof}
Let $A=(a_{ij})$ be the matrix of exponents associated to $\mathcal L$.
We will show that if $\mathcal L$ is quasismooth then the matrix $A$, after 
reordering its rows, satisfies
\begin{enumerate}[1.]
\item $a_{ii}>1$ for all $i$,
\item for all $i_0$, $\sum_{j\neq i_0} a_{i_0 j}\leq 1$,
\item for all $j_0$, $\sum_{i\neq j_0} a_{i j_0}\leq 1$.
\end{enumerate}
An easy algorithm thus gives the necessary part of the proof.
The sufficiency part is obvious.

Let $\gamma=\{0,\dots, n\}\backslash \{i_0\}$. By Corollary \ref{cor:wps} 
there exists a row $j$ with $\sum_{i\not=i_0} a_{ji}\leq 1$.
By the hypothesis on the degree, this implies that $a_{ji_{0}}>1$ 
and this is the only row having an entry bigger than $1$ in position $i_0$.
This gives $1.$ and $2.$ up to reordering the rows of $A$.

Now assume that $\sum_{i\neq j_0} a_{ij_0}> 1$ for some index $j_0$.
This implies that $a_{rj_0}=a_{sj_0}=1$ for some $r,s\not=j_0$ 
and that $j_0$ is the only column containing $1$ in both positions $r$ and $s$
by 2.
This contradicts Corollary \ref{cor:wps} when $\gamma=\{0,\dots,n\}\backslash \{r,s\}$.
\end{proof}

\begin{remark}
Since fake weighted projective spaces are finite quotients of weighted projective spaces 
satisfying the hypothesis of  Proposition \ref{fq} by \cite[Theorem 6.4]{Bu}, 
then quasismoothness in a fake weighted projective space can be checked in its 
weighted projective space covering.
\end{remark}

\section{Quasismooth Calabi-Yau hypersurfaces}\label{sec-cy}

The concept of quasismoothness appeared in the literature 
on Calabi-Yau varieties since it provides a sufficient condition 
to have good singularities.  We recall that a normal projective variety $Y$ of 
dimension $n$ is a {\em Calabi-Yau variety} if it has canonical singularities, 
$K_Y \cong \Osh_Y$ and $h^i(Y,\Osh_Y ) = 0$ for $0 < i < n$.
Moreover, a hypersurface $Y$ of a projective toric variety $X$ is called 
{\em well-formed} if ${\rm codim}_Y(Y\cap {\rm Sing}(X))\geq 2$, 
where ${\rm Sing}(X)$ is the singular locus of $X$.

\begin{proposition}
Let $Y$ be an anticanonical hypersurface of a projective toric variety $X$. 
If $Y$ is quasismooth and well-formed, then $Y$ is a Calabi-Yau variety. 
\end{proposition}

\begin{proof}
 By \cite[Proposition 2.12]{ACG}, $Y$ has canonical singularities and $K_Y \cong \Osh_Y$. 
Moreover the exact sequence of sheaves
 \[
 \xymatrix{
 0\ar[r] & \Osh_{X}(K_X)\ar[r] & \Osh_X\ar[r] & \Osh_Y\ar[r]& 0,
 }
 \]
 gives the exact sequence
 \[
 \xymatrix{
 \dots\ar[r] & H^i(X,\Osh_X)\ar[r] & H^i(Y,\Osh_Y)\ar[r] &H^{i+1}(X,\Osh_X(K_X))\ar[r] &\dots. 
 }
 \]
Since $h^i(X,\Osh_X)=0$ for $i>0$ and 
$h^{i+1}(X,\Osh_X(K_X))=h^{\dim(X)-i-1}(X,\Osh_X)$ for $i<\dim(X)-1$ 
by Serre-Grothendieck duality, we have that 
$h^i(Y,\Osh_Y)=0$ for $0<i<\dim(Y)$.
Thus $Y$ is a Calabi-Yau variety.
\end{proof}

A natural question, which is the original motivation of the present work,
is how quasismoothness behaves with respect to known dualities between families 
of Calabi-Yau varieties. 
For example, we recall the following fact, which is the basis of the 
Berglund-H\"ubsch-Krawitz duality \cite{BH, Kr}.

\begin{proposition}
Let $\mathcal L$ be a monomial linear system 
of Delsarte type in a weighted projective space $X$  
with matrix of exponents $A$.
There exist a weighted projective space  $X'$
and a  monomial linear system  $\mathcal L'$ on $X'$ 
whose matrix of exponents is $A^T$. 
Moreover, $\mathcal L$ is quasismooth if and only if the same holds for $\mathcal L'$.
\end{proposition}

\begin{proof}
The existence of $X'$ and $\mathcal L'$ is constructive, see \cite[Section 4.1]{ACG}.
Quasismoothness of $\mathcal L'$ follows from (the proof of) Corollary \ref{cor-simpl}.
\end{proof}

In \cite{ACG} Berglund-H\"ubsch-Krawitz duality has been generalized  
to give a duality between pairs of polytopes. We recall the main definitions and results.
Given a lattice $M$ and a polytope $P\subset M_{\qq}$ containing the origin in its interior, we will denote by 
$P^*$ the polar of $P$ as in \cite[\S 2.2]{CLS}.

\begin{definition}  A polytope $P\subset M_{\qq}$ with vertices in $M$ 
is \emph{canonical} if ${\rm Int}(P)\cap M=\{0\}$.
A pair of polytopes $(P_1,P_2)$ with $P_1\subseteq P_2\subset M_{\qq}$ is  a \emph{good pair} 
if $P_1$ and $P_2^*$ are canonical.
\end{definition}

To any good pair $(P_1,P_2)$ we can associate a monomial linear system $\mathcal L_1$ in a toric variety $X$ 
as follows. Let $X=X_{P_2}$ be the toric variety defined by the normal fan of $P_2$.
Then $P_2$ is the anticanonical polytope of $X$ and its lattice points give a monomial basis of the anticanonical linear system $|-K_X|$.
The polytope $P_1$ defines a monomial linear subsystem $\mathcal L_1$ of $|-K_X|$ generated by the monomials 
corresponding to its lattice points. Moreover the following holds.

\begin{theorem}\cite[Theorem 1]{ACG}\label{thm1ACG}
Let $(P_1,P_2)$ be a good pair of polytopes and let $X=X_{P_2}$ be the 
toric variety defined by the normal fan of $P_2$.
Then $X$ is a $\qq$-Fano toric variety and the general element of the 
monomial linear system associated to $P_1$ is a Calabi-Yau variety.
\end{theorem}

\begin{remark}
Observe that the polytope $P_1$ in the statement of Theorem \ref{thm1ACG} is not exactly
the Newton polytope $\Delta(\mathcal L_1)$ of the monomial linear system as defined in Section \ref{sec-pre} 
but $\Delta(\mathcal L_1)=\alpha(P_1)+(1,\dots,1)$, where $\alpha$ is the dual of the map $\zz^r\to N$, sending $e_i$ to the primitive generator of the $i$-th ray of the fan of $X$.\end{remark}

The following shows that quasismooth monomial linear systems 
in $\qq$-Fano toric varieties give rise to good pairs.

 \begin{proposition} \label{gp}
Let $P_2\subset M_\qq$ be a polytope such that $P_2^*$ is canonical and 
let $P_1\subseteq P_2$ be a lattice polytope.
If the monomial linear system $\mathcal L_1$ associated to $P_1$ 
is quasismooth, then $(P_1,P_2)$ is a good pair (i.e. $P_1$ is a canonical polytope).
 \end{proposition}
 
 \begin{proof}
By \cite[Corollary 1.6]{ACG}  it is enough to prove 
that the origin is an interior point of $P_1$.
Assume the contrary, i.e. that the origin is contained 
in a facet of $P_1$. Then there exists $n\in N_{\qq}$ 
such that $(n,m)\geq 0$ for all $m\in P_1$.
The vector $n$ is contained in a cone over the faces 
of $P_2^*$, thus we can write 
$n=\sum_{i\in I}\alpha_i\rho_i$ with $\alpha_i\in \qq$ 
positive coefficients.
Thus
\[
(n,m)\geq 0\Leftrightarrow \sum_{i\in I}\alpha_i(m,\rho_i)\geq 0\Leftrightarrow \sum_{i\in I}\beta_i(m,\rho_i)\geq 0,
\]
 where $\beta_i$ are positive integers. 
 We recall that $(m,\rho_i)+1=a_i$ is the exponent of $x_i$ in the monomial corresponding to $m$.
 Thus, if $\beta=\max\{\beta_i\}_{i\in I}$, the above inequality gives
 \[
  \sum_{i\in I}a_i\geq \frac{1}{\beta}\sum_{i\in I}\beta_i.
 \]
 Observe that $D_I=\{x_i=0: i \in I\}$ is contained in the base locus of 
  $\mathcal L_1$ since 
 the above inequality implies that at least one of the exponents 
 $a_i, i\in I$ is positive in any monomial of the basis of $\mathcal L_1$.
 If $|I|=1$, then $\mathcal L_1$ contains the divisor $x_i=0$ 
 in its base locus, thus clearly it is not quasismooth.
On the other hand, if $|I|>1$, the right hand side of the inequality is bigger than $1$.
 This implies that  $\mathcal L_1$  is not quasismooth by Corollary \ref{empty}.
 \end{proof}

A good pair $(P_1,P_2)$ has a natural dual $(P_2^*,P_1^*)$, which is still a good pair.
This gives rise to a duality between the corresponding  linear systems $\mathcal L_1$ 
and $\mathcal L_2^*$  of Calabi-Yau varieties on $X_{P_2}$ and $X_{P_1^*}$ respectively.
 In \cite[Theorem 2]{ACG} it is proved that, in case $P_1,P_2$ are both simplices, this 
duality is exactly Berglund-H\"ubsch-Krawitz duality.
With this in mind, we asked ourselves: is quasismoothness preserved by this duality? More precisely:  
does $\mathcal L_1$ quasismooth imply $\mathcal L_2^*$ quasismooth?
Unfortunately the answer is no, as the following example shows.

\begin{example}

Let $X=\pp^2\times\pp^1$ and let $P_2$ be its anticanonical polytope. Let $\mathcal L_1$ be the quasismooth linear system as in 
 Example \ref{ex1}. Thus $(\Delta(\mathcal L_1),P_2)$ is a good pair by Proposition \ref{gp}.
The toric variety $X_{\Delta(\mathcal L_1)^*}$ has Cox ring with variables $y_1,\dots, y_8$ 
with integer grading given by the matrix
\[
\begin{pmatrix*}[c]
0 & 1 & 0 & 2& 3& 0& 0& 0\\
0& 1& 1& 0& 0& 1& 1& 0\\
0& 1& 1& 0& 1& 0& 0& 1\\
0& 1& 1& 1& 2& 1& 0& 0\\
1& 0& 0& 1& 0& 1& 1& 0
 \end{pmatrix*},    
 \]   
 quotient grading $1/2( 0, 0, 0, 0, 0, 0, 1, 1 )$ and 
the components of its irrelevant ideal are:
 \[
 \begin{array}{c}
   (y_7, y_6, y_4), (y_5, y_4), (y_6, y_3), (y_8, y_5, y_3), (y_7, y_2),\\ 
        (y_8, y_5, y_2), (y_8, y_3, y_2), (y_5, y_3, y_2), (y_8, y_1),\\ 
        (y_7, y_6, y_1), (y_7, y_4, y_1), (y_6, y_4, y_1).
        \end{array}
\]
 The dual  pair $(P_2^*,\Delta(\mathcal L_1)^*)$ gives a  monomial linear system $\mathcal L_2^*$ generated by
\[
y_2^2y_4^2y_6^2y_8^2,\
y_3^3y_4^3y_7y_8, \
y_5^2y_6^2y_7^2y_8^2,\
y_1^3y_2^3y_5y_6, \
y_1^2y_3^2y_5^2y_7^2.
\]
One can observe that $\mathcal L_2^*$ is not quasismooth at the point $(1, 1, 1, 1, 0, 0, 0, 0)$.
\end{example}

\section{Applications in low dimension}\label{sec-app}

\begin{theorem} A  monomial linear system $\mathcal L$ of 
curves in a projective toric surface $X$ with monomial basis $S$ 
is quasismooth if and only if either $\mathcal L$ is base point free or 
 the following hold
\begin{enumerate}
\item if $\{x_i=0\}\subseteq B^*(\mathcal L)$ then  $S$ contains a unique monomial where $x_i$ appears with exponent one;
\item if $\{x_i=x_j=0\}\subseteq B^*(\mathcal L)$ then $S$ contains a monomial $x^a$ where $a_i+a_j=1$.
\end{enumerate}
\end{theorem}

\begin{proof}
Assume $\{x_i=0\}\subseteq B^*(\mathcal L)$ and let $\rho_i$ be the ray of the fan of $X$ corresponding to $x_i$. 
By Theorem \ref{thm1} the polytope $\Delta_{\rho_i}^{\rho_i}(\mathcal L)$ must be a point, i.e. 
$S$ contains a unique monomial where $x_i$ appears with exponent one.

If $\{x_i=x_j=0\}\subseteq B^*(\mathcal L)$ and $\sigma=\langle \rho_i,\rho_j\rangle$,
the dimension of $\Delta_\sigma^{\rho_i},\Delta_\sigma^{\rho_j}$ is at most 0 by Remark \ref{ambient}.
By Theorem \ref{thm1} one of them has to be not empty. 
Thus there must be a monomial in $S$ 
containing one variable among $x_i,x_j$ with exponent one 
and not containing the other one.
\end{proof}

Observe that a quasismooth curve is smooth,
 since quasismooth implies normal (see the proof of \cite[Proposition 2.4]{ACG}).
However the converse is false, as the following example shows.

\begin{example}
Let $X=\pp(2,3,5)$, whose fan can be taken to have rays $e_1,e_1+5e_2, -e_1-3e_2$,
and let $Y$ be the curve defined by $x^3-y^2=0$.
Clearly the curve is not quasismooth. In order to prove that $Y$ is smooth, it is enough 
to check smoothness at the point $(0,0,1)$. 
The affine chart containing such point is $U_{\sigma}$, 
where $\sigma={\rm cone}(e_1,e_1+5e_2)$, thus it is given by the closure of the 
image of the map
\[
\varphi_{\sigma}:(\cc^*)^2\to \cc^3,\ (t_1,t_2)\mapsto (t_1,t_2, t_1^5t_2^{-1}).
\]
An easy computation shows that the curve $Y$ is the closure of the image of the 
one parameter subgroup $t\mapsto (t,t^3)$, 
thus the curve $Y\cap U_{\sigma}$ is the closure of the image of the map $t\mapsto (t,t^3,t^2)$,
which is smooth. 
\end{example}

\begin{theorem} A  monomial linear system $\mathcal L$ of 
surfaces in a simplicial projective toric threefold $X$ with monomial basis $S$ 
is quasismooth if and only if either $\mathcal L$ is base point free or 
 the following hold
\begin{enumerate}
\item if $\{x_i=0\}\subseteq B^*(\mathcal L)$ then  $S$ contains a unique monomial  where $x_i$ appears with exponent one;
\item if $\{x_i=x_j=0\}\subseteq B^*(\mathcal L)$ then 
\begin{itemize}
\item either $S$ contains a monomial $x^{a}$ with $a_i=1, a_j=0$
and a monomial $x^b$ with $b_i=0,b_j=1$, 
\item or 
the exponent of $x_i$ is either $0$ or $\geq 2$ in each monomial of $S$ and
there exists a unique monomial $m=x^a$ with $a_i=0, a_j=1$. 
\end{itemize}
\item if $\{x_i=x_j=x_k=0\}\subseteq B^*(\mathcal L)$ then $S$ contains a monomial $x^a$ where $a_i+a_j+a_k=1$.
\end{enumerate}
\end{theorem}

\begin{proof}
Assume $\{x_i=0\}\subseteq B^*(\mathcal L)$ and let $\rho_i$ be the ray of the fan of $X$ corresponding to $x_i$. 
By Theorem \ref{thm1} the polytope $\Delta_{\rho_i}^{\rho_i}(\mathcal L)$ must be a point, i.e.
$S$ contains a unique monomial where $x_i$ appears with exponent one.

If $\{x_i=x_j=0\}\subseteq B^*(\mathcal L)$ and $\sigma=\langle \rho_i,\rho_j\rangle$,
we distinguish whether $\Delta_\sigma^{\rho_i},\Delta_\sigma^{\rho_j}$ are both not empty
or one of them is empty, say $\Delta_\sigma^{\rho_i}=\emptyset$.
The first situation implies the existence of a monomial containing $x_i$ with exponent one and not containing $x_j$ 
and of a monomial containing $x_j$ with exponent one and not containing $x_i$. 
In the second case by Theorem \ref{thm1} the only non-empty polytope $\Delta_\sigma^{\rho_j}$ 
must be a point, i.e. $S$ contains a unique monomial where $x_j$ appears with exponent one.

If $\{x_i=x_j=x_k=0\}\subseteq B^*(\mathcal L)$ and $\sigma=\langle \rho_i,\rho_j,\rho_k\rangle$, 
by Theorem \ref{thm1} one of the polytopes $\Delta_{\sigma}^{\rho_i}(\mathcal L)$, which are either points or empty by Remark \ref{ambient},
must be not empty.
\end{proof}

 \bibliographystyle{plain}
\bibliography{Biblioqs}

\end{document}